\documentclass{amsart}
\usepackage{amsmath,amsthm,amssymb,amsfonts}
\newtheorem{theorem}{Theorem}

\newtheorem{corollary}[theorem]{Corollary}

\newtheorem{lemma}[theorem]{Lemma}
\newtheorem{proposition}[theorem]{Proposition}
\newtheorem{remark}[theorem]{Remark}
\usepackage{enumitem}
\begin{document}

\title[The reachable space of the heat equation]{The reachable space of the heat equation for a finite rod as a Reproducing Kernel Hilbert Space}

\author{Marcos L\'{o}pez-Garc\'{\i}a}
\thanks{
The author was partially supported by project PAPIIT IN100919 of DGAPA, UNAM}
\email{marcos.lopez@im.unam.mx}
\address{
Instituto de Matem\'{a}ticas-Unidad Cuernavaca \\
   Universidad Nacional Aut\'{o}noma de M\'{e}xico\\
   Apdo. Postal 273-3, Cuernavaca Mor. CP 62251, M\'exico}

\maketitle
\begin{abstract}
We use some results from the theory of Reproducing Kernel Hilbert Spaces to show that the reachable space of the heat equation for a finite rod with either one or two Dirichlet boundary controls is a RKHS of analytic functions on a square, and we compute its reproducing kernel. We also show that the null reachable space of the heat equation for the half line with Dirichlet boundary data is a RKHS of analytic functions on a sector, whose reproducing kernel is (essentially) the sum of pullbacks of the Bergman and Hardy kernels on the half plane $\mathbb{C}^+$. We also consider the case with Neumann boundary data.
\end{abstract}

\section{Introduction}
Let $T>0$ fixed.  Consider the following control system
\begin{eqnarray}\label{principal}
\partial_t w-\partial_{xx} w=0, && 0<x<1, \, 0<t<T,  \notag  \\ 
w(0,t)=u_{\ell}(t), \quad w(1,t)=u_r(t), && 0<t<T ,\\
w(x,0)=0, &&  0<x<1,\notag
\end{eqnarray}
which models the temperature propagation in a rod with length 1.\\

In control theory is an important issue to describe the so-called null reachable space, at time $T>0$, defined as follows
$$\mathcal{R}_T:=\{ w(\cdot,T): w \text{ is solution of system (\ref{principal}) with controls }u_{\ell}, u_{r}\in L^2_{\mathbb{C}}(0,T)\}.$$
Using the null controllability of the system (\ref{principal}) in any positive time (see \cite[Theorem 3.3]{fatto}), one can show that the set of states $w(\cdot,T)$ reached by solutions of system (\ref{principal}) from any initial datum $w(x,0)\in L^2(0,1)$ coincides with $\mathcal{R}_T$. The null controllability also implies that $\mathcal{R}_T$ does not depend on $T>0$ (see a proof in \cite{hart}), thus $\mathcal{R}$ will denote this space.\\

The problem is to identify the space of all analytic extensions of the functions in (some subspace of) $\mathcal{R}$ in terms of spaces of analytic functions with some structure.\\

From \cite[Theorem 1.1]{darde} and \cite[Theorem 2.1]{martin} we have that 
$$hol(\overline{Q}) \subset \mathcal{R} \subset hol(Q),$$
where $Q=\{(x,y)\in \mathbb{R}^2: |y|<x, \, |y|<1-x \}$ and $hol(\overline{Q})$ is the set of all analytic functions on a neighborhood of $\overline{Q}$. Hence these results established the domain of analyticity to deal with.\\

In \cite{hart} was proved that
$$E^2(Q)\subset \mathcal{R} \subset A^2(Q),$$
where $E^2(Q)$ is the Hardy-Smirnov space on $Q$ and $A^2(Q)$ is the (unweighted) Bergman space on $Q$. Thus, in this work well-known spaces of analytic functions with some structure appeared for the first time.\\

In \cite{orsini} the author has proved that the null reachable space $\mathcal{R}$ is the sum of two Bergman spaces on sectors (whose intersection is Q), i.e.
$$\mathcal{R}=A^2(\Delta)+A^2(1-\Delta),$$
where $\Delta:=\{z\in \mathbb{C}:|\arg z|<\pi/4\}$.\\

As the author has remarked, given any function $f\in hol(Q)$, how can we write $f$ as a sum of two functions in those different Bergman spaces? Notice that $A^2(\Delta)$ and $A^2(1-\Delta)$ are RKHSs on different domains.\\

In this work our main result (Theorem \ref{masmenos}) shows that the null reachable space $\mathcal{R}$ is the sum of two RKHSs on the same domain $Q$, therefore $\mathcal{R}$ is a RKHS on $Q$ (by properties of the RKHSs). Corollary 2.1 in \cite[page 98]{saitohnuevo} gives a necessary and sufficient condition so that a function $f\in hol(Q) $ can be in $\mathcal{R}$.

\section{Statement of the results}
To get the characterization of $\mathcal{R}$ as a RKHS on $Q$,
we proceed in several steps. We characterize the subspaces in $\mathcal{R}$ corresponding to the cases either $u_r=0$ or  $u_\ell=0$ or $u_r=-u_\ell$ or $u_r=u_\ell$ as RKHSs of analytic functions on a square.\\

In \cite{darde} the authors ask for a characterizacion of the null reachable space, at time  $T>0$, with just one Dirichlet boundary control. So we consider the null reachable space, at time  $T>0$,
$$\mathcal{R}_T^{\ell}:=\{ w(\cdot,T): w \text{ is solution of system (\ref{principal}) with }u_{\ell}\in L^2_{\mathbb{C}}(0,T), u_{r}\equiv 0\}$$
with one Dirichlet boundary control on the left. As before, we can see that the space $\mathcal{R}_T^{\ell}$ does not depend on $T>0$, so $\mathcal{R}^\ell$ will denote this space.\\

Motivated by the idea in \cite{hart} of writing the solution of system (\ref{principal}) in terms of integral operators having well known heat kernels (see \ref{opintegral}), and by using the characterization of the image of a linear mapping as a RKHS (see \cite[page 134]{saitohnuevo}) we have obtained the characterization of $\mathcal{R}^\ell$ as a RKHS on a square.\\

First, we introduce some notation and definitions. Consider the square $D:=\left\{(x,y)\in \mathbb{R}^2:|y| <x, |y| <2-x \right\}$, the open set $D^\ell_\infty:=\bigcup_{n\in \mathbb{Z}}(2n+D)$ and the following positive definite function on the sector $\Delta$,
$$\mathcal{K}_0(z,w;T):= \frac{z\overline{w}}{\pi}  e^{-\frac{z^2+\overline{w}^2}{4T}}\left(\frac{1}{ (z^2+\overline{w}^2)^2} + \frac{1}{ 4 T\left(z^2+\overline{w}^2\right)}\right), \quad z,w \in \Delta.$$

\begin{theorem}\label{main}
For each $T>0$ fixed,  we have that 
$$\mathcal{R}^\ell=\{f\in hol(D^\ell_\infty):f(z+2)=f(z)= -f(-z), f|_D\in \mathcal{H}^\ell_T(D) \}$$
where $\mathcal{H}^\ell_T(D)$
is the RKHS of analytic functions on $D$ with reproducing kernel
\begin{equation}\label{kernelazo}
\mathcal{K}_\ell(z,w;T):= \sum_{m,n\in \mathbb{Z}}\mathcal{K}_0(z+2n,w+2m;T), \quad z,w \in D.
\end{equation}
The space $\mathcal{H}^\ell_T(D)$ is endowed with the norm given in (\ref{norma1}).
\end{theorem}

Notice that the properties of 2-periodicity, oddness and analyticity domain of $f\in \mathcal{R}^\ell$ are inherited from those of the analytic extension of the heat kernel $(\partial_x\theta)(x,t)$, see Remarks \ref{perimpar}, \ref{paraprincipal}. \\

We also have the corresponding result for the null reachable space with one Dirichlet boundary control on the right, defined as follows
$$\mathcal{R}_T^r:=\{ w(\cdot,T): w \text{ is solution of system (\ref{principal}) with } u_{\ell}\equiv 0,u_{r}\in L^2_{\mathbb{C}}(0,T)\}.$$
\begin{theorem}\label{derecho}
Let $D^r_{\infty}=-1+D^\ell_{\infty}$. For each $T>0$ fixed,  we have that 
$$\mathcal{R}^r=\{f\in hol(D^r_\infty):f(z+2)=f(z)=- f(-z), f|_{-1+D}\in \mathcal{H}^r_T(-1+D) \}$$
where $\mathcal{H}^r_T(-1+D)$
is the RKHS of analytic functions on $-1+D$ with reproducing kernel
\begin{equation}\label{derechazo}
\mathcal{K}_r(z,w;T)= \sum_{m,n\in \mathbb{Z}}\mathcal{K}_0(z+2n+1,w+2m+1;T), \quad z,w \in -1+D.
\end{equation}
The space $\mathcal{H}^r_T(-1+D)$ is endowed with the norm given in (\ref{norma1.5}).
\end{theorem}
By (\ref{opintegral}) and Theorem \ref{main} we also have
$$\mathcal{R}^r=\{f\in hol(D^r_\infty): f(z+2)=f(z), f(-z)=-f(z), f(\cdot -1)|_{D}\in \mathcal{H}_\ell^T(D) \}.$$

\bigskip

Our approach also provides the characterization of the following subspaces in $\mathcal{R}$,
$$\mathcal{R}^+_T:=\{ w(\cdot,T): w \text{ is solution of system (\ref{principal}) with }u_{\ell}\in L^2_{\mathbb{C}}(0,T), u_{r}=-u_\ell\},$$
$$\quad \text{and } \mathcal{R}^-_T:=\{ w(\cdot,T): w \text{ is solution of system (\ref{principal}) with }u_{\ell}\in L^2_{\mathbb{C}}(0,T), u_{r}=u_\ell\}.$$

\bigskip

Once again, (2.14), (2.15) and Theorem 3.3 in \cite{fatto} imply the null controllability of system (\ref{principal}) with initial datum $w(x,0)\in L^2(0,1)$ and the controls $u_\ell, u_r\in L^2_{\mathbb{C}}(0,T)$ satisfying either $u_r=-u_\ell$ or $u_r=u_\ell$. Thus, the spaces $\mathcal{R}^+_T,\mathcal{R}^-_T$ do not depend on $T>0$.
\begin{theorem}\label{laterales}
Let $Q_\infty:=\bigcup_{n\in \mathbb{Z}}(n+Q)$. For each $T>0$ fixed we have that
\begin{enumerate}
\item $\mathcal{R}^+=\{f\in hol(Q_\infty): f(z+1)=f(z)=-f(-z), f|_Q\in \mathcal{H}^+_T(Q) \}$
where $\mathcal{H}^+_T(Q)$ is the RKHS on $Q$ with reproducing kernel 
\begin{eqnarray*}
\mathcal{K}_+(z,w;T)&:=&\sum_{m,n\in \mathbb{Z}}\mathcal{K}_0(z+n,w+m;T)\\
&=&\mathcal{K}_\ell(z,w;T)+\mathcal{K}_\ell(z+1,w+1;T)+\mathcal{K}_\ell(z+1,w;T)+\mathcal{K}_\ell(z,w+1;T),
\end{eqnarray*}
for $z,w\in Q$. The space $\mathcal{H}^+_T(Q)$ is endowed with the norm given in (\ref{norma2}).
\item $\mathcal{R}^-=\{f\in hol(Q_\infty): -f(z+1)=f(z)=- f(-z), f|_Q\in \mathcal{H}^-_T(Q) \}$
where $\mathcal{H}^-_T(Q)$ is the RKHS on $Q$ with reproducing kernel 
$$\mathcal{K}_-(z,w;T):=\mathcal{K}_\ell(z,w;T)+\mathcal{K}_\ell(z+1,w+1;T)-\mathcal{K}_\ell(z+1,w;T)-\mathcal{K}_\ell(z,w+1;T),$$           
 for $z,w\in Q$. The space $\mathcal{H}^-_T(Q)$ is endowed with the norm given in (\ref{norma3}).
\end{enumerate}
\end{theorem}
As a consequence, we get a description of the null reachable space $\mathcal{R}$.
\begin{theorem}\label{masmenos}
We have that 
$$\mathcal{R}=\mathcal{R}^++\mathcal{R}^- \text{ and }\mathcal{R}^+\cap \mathcal{R}^- =\{0\}.$$ 
Moreover, $\mathcal{R}|_Q:=\{f|_Q:f\in \mathcal{R}\}=\mathcal{H}^+_T(Q)+\mathcal{H}^-_T(Q)$ with reproducing kernel
$$\mathcal{K}_\ell(z,w;T)+\mathcal{K}_\ell(z+1,w+1;T),\quad z,w\in Q.$$
The space $\mathcal{R}|_Q$ is endowed with the norm given in (\ref{sumaesp}).
\end{theorem}
Clearly, the condition $\mathcal{R}^+\cap \mathcal{R}^- =\{0\}$ follows from the functional equations in the last result.
\bigskip

Next we consider the case with Neumann boundary data.
\begin{eqnarray}\label{flujo}
\partial_t v-\partial_{xx} v=0, && 0<x<\infty, \, 0<t<T,  \notag  \\ 
(\partial_x v)(t,0)=u_\ell(t),\quad (\partial_x v)(t,1)=u_r(t), && 0<t<T ,\\
v(0,x)=0, && 0<x<\infty.\notag
\end{eqnarray}
We set 
$$\mathcal{R}_T^N:=\{ v(\cdot,T): v \text{ is solution of system (\ref{flujo}) with }u_\ell,u_r\in L^2_{\mathbb{C}}(0,T)\}.$$
\begin{corollary}\label{neumann}
We have that $\mathcal{R}^N_T$ does not depend on $T>0$, and
$$\mathcal{R}^N=\{f\in hol(Q_\infty):f'\in \mathcal{R} \}.$$
\end{corollary}

\bigskip

In some situations, the null reachable space at time $T>0$ of a certain heat equation can be described in terms of well known analytic functions spaces. For instance, consider the heat equation for the half line,
\begin{eqnarray}\label{basico}
\partial_t v-\partial_{xx} v=0, && 0<x<\infty, \, 0<t<T,  \notag  \\ 
v(0,t)=u(t), && 0<t<T ,\\
v(x,0)=0, && 0<x<\infty.\notag
\end{eqnarray}
Its corresponding null reachable space at time $T>0$ is given by
$$\mathcal{R}_T^q:=\{ v(\cdot,T): v \text{ is solution of system (\ref{basico}) with }u\in L^2_{\mathbb{C}}(0,T)\}.$$

As usual, $\Re z$, $\Im z$ denote the real and the imaginary parts of $z\in \mathbb{C}$. 
Let $\mathbb{C}^+=\{z\in \mathbb{C}: \Re z >0\}$ be the positive half space. The following result characterizes the null reachable space $\mathcal{R}_T^q$ as a RKNS whose reproducing kernel is (essencially) a sum of pullbacks of the Bergman and Hardy kernels on $\mathbb{C}^+$.
\begin{theorem}\label{iniciando}
We have that $\mathcal{R}_T^q$ does not depend on $T>0$, and
$$\mathcal{R}^q=e^{-z^2} A^2(\Delta)+ze^{-z^2}\left\{f\circ \varphi | f\in H^2(\mathbb{C}^+)  \right\},$$
where $\varphi(z)=z^2, z\in \Delta$, $A^2(\Delta)$ is the unweighted Bergman space on $\Delta$  and $H^2(\mathbb{C}^+)$ is the Hardy space on $\mathbb{C}^+$. The space $\mathcal{R}^q$ is endowed with the norm given in (\ref{sumanorm}).
\end{theorem}

This paper is organized as follows. In the next section we introduce notation, give some results about RKHSs, and make the computations needed to prove the results. In Section \ref{probando} we provide the proofs.

\section{Preliminaries}
In this section we use some results about the one dimensional heat equation that can be found in \cite{cannon}. First, consider the heat kernel on the upper half plane $\mathbb{R}^2_+:=\{(x,t)\in \mathbb{R}^2: t>0 \}$ given as follows
$$K(x,t):=\frac{1}{\sqrt{4\pi t} }e^{-\frac{x^2}{4t}}, \quad (x,t)\in \mathbb{R}^2_+.$$
In order to describe the solution $w(x,t)$ of system (\ref{principal}) we introduce the so-called theta function 
$$\theta(x,t):=\sum_{n\in \mathbb{Z}}K(x+2n,t), \quad (x,t)\in \mathbb{R}^2_+,$$
so we have the system (\ref{principal}) admits a unique solution $w\in C([0,\infty),W^{-1,2}(0,1))$ given by (see \cite[Theorem 6.3.1]{cannon}, \cite{hart})
\begin{equation}\label{opintegral}
w(x,t)=-2\int_0^t (\partial_x \theta)(x,t-\tau)u_{\ell}(\tau)d\tau+2\int_0^t (\partial_x \theta)(x-1,t-\tau)u_r(\tau)d\tau.
\end{equation}

For an open set $\Omega \subset \mathbb{C}$, the (unweighted) Bergman space on $\Omega$ is the vector space
$$A^2(\Omega):=\{f: \Omega  \rightarrow \mathbb{C} | f \text{ analytic on }\Omega \text{ and } f\in L^2(\Omega)\}$$
endowed with the inner product
$$\langle f,g\rangle_{A^2(\Omega)}:=\frac{1}{\pi}\int_{\Omega}f(z)\overline{g(z)}dxdy.$$
We also consider the Hardy space on the half space $\mathbb{C}^+$,
$$H^2(\mathbb{C}^+):=\left\{f:\mathbb{C}^+\rightarrow \mathbb{C}|f \text{ analytic on }\mathbb{C}^+ \text{ and } \sup_{x>0}\int_{-\infty}^{\infty}|f(x+iy)|^2dy <\infty\right\}$$
endowed with the inner product
$$\langle f,g\rangle_{H^2(\mathbb{C}^+)}:=\sup_{x>0}\int_{-\infty}^{\infty}f(x+iy)\overline{g(x+iy)} dy.$$

Consider the following positive definite functions on $\Delta$
$$\mathcal{K}_1(z,w):=\frac{4z\overline{w}}{ (z^2+\overline{w}^2)^2}, \quad  \mathcal{K}_2(z,w):=\frac{1}{ z^2+\overline{w}^2}, \quad z,w\in \Delta,$$
and the biholomorphism $\varphi(z)=z^2$ from $\Delta$ onto $\mathbb{C}^+$.\\

\begin{remark}\label{pullback}
Notice that $\mathcal{K}_1(z,w)=\mathcal{K}_B(\varphi(z),\varphi(w))\varphi '(z)\overline{\varphi '(w)}$ where $\mathcal{K}_B(z,w)$ is the reproducing kernel for the Bergman space $A^2(\mathbb{C}^+)$, so that $\mathcal{K}_1(z,w)$ is the reproducing kernel for the Bergman space $A^2(\Delta)$ (see \cite[page 12]{duren}).
\end{remark}

The following result shows that $\mathcal{K}_2(z,w)$ is the reproducing kernel for the pullback space, induced by the function $\varphi$, of the Hardy space $H^2(\mathbb{C}^+)$.
\begin{lemma}\label{hardpull}
$\mathcal{K}_2(z,w)$ is the reproducing kernel for the RKHS $\mathcal{H}_{\varphi}(\Delta):=\{f\circ\varphi:f\in H^2(\mathbb{C}^+)\}$.
\end{lemma}
\begin{proof}
Here $\mathcal{K}_H$ stand for the reproducing kernel for $H^2(\mathbb{C}^+)$. Let $ev_z:H^2(\mathbb{C}^+)\rightarrow \mathbb{C}$ be the functional evaluation at $z\in \mathbb{C}^+$. If $g\in \bigcap_{p\in \Delta}\ker(ev_{\varphi(p)})$ then $g\equiv 0$, so Theorem 2.9 in \cite[page 81]{saitohnuevo} implies that the RKHS with reproducing kernel $\mathcal{K}_H(\varphi(z),\varphi(w))=\mathcal{K}_2(z,w)$ is the space $\mathcal{H}_{\varphi}(\Delta)$ equipped with the inner product
\begin{equation}\label{normpull}
\langle f\circ \varphi,g\circ \varphi\rangle_{\mathcal{H}_{\varphi}(\Delta)}=\langle f,g\rangle_{H^2(\mathbb{C}^+)}.
\end{equation}
 \end{proof}
 For each $t>0$ fixed, consider the entire function
$$(\partial_x K) (z,t):=-\frac{z}{4\sqrt{\pi}t^{3/2}}e^{-\frac{z^2}{4t}},\quad z\in \mathbb{C},$$ 
which is the analytic extension of the function $(\partial_x K) (x,t)$, $x\in \mathbb{R}$.
\begin{lemma}\label{estheta}
\begin{enumerate}[label=\roman*)]
\item For each $t>0$ fixed, the function
\begin{equation}\label{corazon}
(\partial_x\theta)(z,t):=\sum_{n \in \mathbb{Z}}  (\partial_x K) (z+2n,t)
\end{equation}
is holomorphic on $D$ and continuous on $\overline{D}$.
\item For each compact set $\mathcal{F}\subset D$ and $t>0$, there exists a constant $C_{\mathcal{F},t}>0$ such that
\begin{equation}\label{normafinita}
\sum_{n \in \mathbb{Z}}  \left(\int_0^t  \left |( \partial_x K )(z+2n,t-\tau) \right |^2  d\tau \right)^{1/2} \leq C_{\mathcal{F},t} \quad\text{for all } z\in \mathcal{F}.
\end{equation}
\end{enumerate}
\end{lemma}
\begin{proof}
For $n\leq -2$ we have that $n^2 \geq  - 2nx$ if $0\leq x \leq 1$, and $3n^2 \geq 4-4x-4nx$ if $1\leq x \leq 2$; therefore
\begin{equation}\label{fundamental}
|z+2n|\left|e^{-\frac{(z+2n)^2}{4t}}\right| \leq \left\{  \begin{array}{ll}
4ne^{-\frac{n^2}{t}}, &  n\geq 1, \quad z\in \overline{D},  \\
2e^{-\frac{\Re(z^2)}{4t}}, &  n=0, \quad z \in D,\\
2  e^{-\frac{\Re((z-2)^2)}{4t}}, &  n=-1, \quad z\in D,\\
4|n|e^{-\frac{n^2}{4t}}, &  n\leq -2, \quad z\in \overline{D}.
\end{array}   \right .
\end{equation}
Since $e^{-s}\leq C_{\sigma}s^{-\sigma}$ for all $s,\sigma >0$, together the Weierstrass M-test imply the series in (\ref{corazon}) converges absolutely and uniformly on $\overline{D}$, and the result  \textit{i)} follows.\\  

For $n\in \mathbb{Z}\backslash \{-1,0\}$, $z\in \overline{D}$ we have
$$\int_0^t  \left |( \partial_x K )(z+2n,t-\tau) \right |^2  d\tau \leq \frac{4}{\pi n^2}\int_{n^2/(2t)}^{\infty}\rho e^{-\rho}d\rho\leq \frac{4}{\pi}\left(1+\frac{1}{t}\right)e^{-\frac{n^2}{2t}}.$$

Let $z_0 \in \mathcal{F}$ be such that $\Re (z_0^2)= \min_{z\in \mathcal{F}} \Re( z^2)$, therefore
$$\int_0^t  \left |( \partial_x K )(z,t-\tau) \right |^2  d\tau  \leq \frac{1}{\pi(\Re (z_0^2))^2}\int_{\Re (z_0^2)/(2t)}^{\infty} \rho e^{-\rho}d\rho = \frac{1}{\pi(\Re (z_0^2))^2}\left(1+\frac{\Re (z_0^2)}{2t}\right)e^{-\frac{\Re(z_0^2)}{2t}}$$
for all $z\in \mathcal{F}$.\\

In a similar way, let $z_1 \in \mathcal{F}$ be such that $\Re ((z_1-2)^2)= \min_{z\in \mathcal{F}} \Re( (z-2)^2)$, therefore
$$\int_0^t  \left |( \partial_x K )(z-2,t-\tau) \right |^2  d\tau  \leq  \frac{1}{\pi(\Re ((z_1-2)^2))^2}\left(1+\frac{\Re ((z_1-2)^2)}{2t}\right)e^{-\frac{\Re((z_1-2)^2)}{2t}}$$
for all $z\in \mathcal{F}$.
\end{proof}

\begin{remark}\label{perimpar}
In fact, by making an easy modification in the last proof we get that $(\partial_x\theta)(\cdot,t)\in hol(D_\infty) $. Clearly, for each $t>0$ fixed the function $(\partial_x \theta)(\cdot,t)$ fulfills the following functional equations
\begin{equation}\label{ecuacionesfunc}
f(z+2)=f(z)=-f(-z), \quad z\in D_\infty.
\end{equation}
\end{remark}

\begin{remark}\label{paraprincipal}
Let $t>0$ fixed and $u\in L^2_{\mathbb{C}}(0,t)$. Lemma \ref{estheta} together Morera and Fubini's theorems imply that the continuous function (by (\ref{fundamental}) and the dominated convergence theorem)
$$z\mapsto \int_0^t u(\tau)(\partial_x \theta)(z,t-\tau)d\tau$$
is holomorphic on $D_\infty$ and satisfies the functional equations in (\ref{ecuacionesfunc}). 
\end{remark}

We write $\overline{D}$ for the closure of the set $D$.
\begin{proposition}\label{convergiendo}
Let $t>0$ fixed. The series introduced in (\ref{kernelazo}) converges absolutely and uniformly on $\overline{D} \times D$ (or $D \times \overline{D}$). Thus $\mathcal{K}_\ell(\cdot,w;t)$ is an analytic function on $D$, and $\mathcal{K}_\ell(z,\cdot;t)$ is an anti-analytic function on $D$.
\end{proposition}
\begin{proof}
For $|n|,|m|\geq 3$ and $z,w\in \overline{D}$ we have that 
$$\left|(z+2n)^2+(\overline{w}+2m)^2\right|\geq 4(n^2+m^2-2|n|-2|m|-2)\geq 16.$$
By using (\ref{fundamental}) and the last inequality we have the series defining $\mathcal{K}_\ell(\cdot,\cdot;t)$ converges absolutely and uniformly on $\overline{D}\times \overline{D}$ whenever we sum over all the indexes satisfying $|n|,|m|\geq 3$.\\

Now suppose that there exist $z\in\overline{D}, w \in D$ such that $(z+2n)^2+(\overline{w}+2m)^2=0$. Then $z+2n=\pm i(\overline{w}+2m)$, thus $|2n+\Re z|=| \Im w|\leq1$, so $n=0$ or $n=-1$. By simmetry, we also have $m=0$ or $m=-1$. If $n=m=0$ then $z=\pm i \overline{w} $, which is a contradiction because $\overline{D}\cap (\pm iD)= \emptyset$. In any case, we get a contradiction because $(\overline{D}-2)\cap (\pm i(D-2))= \emptyset$, $(\overline{D}-2)\cap (\pm iD)= \emptyset$ and $\overline{D}\cap (\pm i(D-2))= \emptyset$. This completes the proof. 
\end{proof}

Let $\mathcal{F}(E)$ be the vector space consisting of all complex-valued functions on a set $E$, and let $(\mathcal{H},\langle \cdot, \cdot\rangle_{\mathcal{H}})$ be a Hilbert space. For a mapping $\mathbf{h}:E\rightarrow \mathcal{H}$, consider the induced linear mapping $\mathbf{L}:\mathcal{H}\rightarrow \mathcal{F}(E)$ defined by
$$\mathbf{L}\mathbf{f}(p)= 􏰔\langle\mathbf{f} , \mathbf{h}(p)\rangle_{\mathcal{H}}.$$

The vector space $\mathcal{R}(\mathbf{L}):=\{\mathbf{L}\mathbf{f}:\mathbf{f}\in \mathcal{H}\}$ is endowed with the norm
$$\|f\|_{\mathcal{R}(\mathbf{L})}=\inf\{\|\mathbf{f}\|_{\mathcal{H}}:\mathbf{f}\in \mathcal{H}, f=\mathbf{L}(\mathbf{f})\}.$$
A fundamental problem about the linear mapping $\mathbf{L}$ is to characterize the vector space $\mathcal{R}(\mathbf{L})$. The following result summarizes Theorems 2.36, 2.37 in \cite[pages 135--137]{saitohnuevo} and provides an answer to the last question.

\begin{theorem}\label{saitochi}
\begin{enumerate}
\item $(\mathcal{R}(\mathbf{L}),\|\cdot\|_{\mathcal{R}(\mathbf{L})})$ is a RKHS with reproducing kernel
$$\mathbf{K}(p,q)=\langle \mathbf{h}(q),\mathbf{h}(p) \rangle_{\mathcal{H}}, \quad p,q \in E.$$
\item The mapping $\mathbf{L}:\mathcal{H}\rightarrow \mathcal{R}(\mathbf{L})$ is a surjective bounded operator with operator norm less than 1.
\end{enumerate}
\end{theorem}

\section{Proofs of the results}\label{probando}
\begin{proof}[Proof of Theorem \ref{main}]
By (\ref{opintegral}) and Remark \ref{paraprincipal} we have that $w(\cdot,T)\in hol(D_\infty)$ and fulfills the functional equations in (\ref{ecuacionesfunc}).\\

Lemma \ref{estheta} \textit{ii)} implies that the function $h:D \rightarrow L^2_{\mathbb{C}}(0,T)$ given by

\begin{equation}\label{achecita}
h_z(t)=\overline{(\partial_x\theta)(z,T-t)}, \quad t \in (0,T),
\end{equation}
makes sense, and Remark \ref{paraprincipal} implies that the linear operator $\mathcal{L}_T^{\ell}:L^2_{\mathbb{C}}(0,T)\rightarrow hol(D)$ given by
$$(\mathcal{L}_T^\ell u)(z)=\langle u, h_z \rangle_{L^2_{\mathbb{C}}(0,T)}, \quad z\in D,$$
is well defined. By (\ref{opintegral}) we have
$$w(z,T)=-2(\mathcal{L}_T^\ell u_\ell)(z), \, z\in D,\quad u_\ell\in L^2_{\mathbb{C}}(0,T).$$

Theorem \ref{saitochi} implies that $\mathcal{H}^\ell_T(D):=\mathcal{R}(\mathcal{L}_T^\ell)$ is a RKHS on $D$ with reproducing kernel
$$\mathcal{K}^*(z,w;T)=\langle h_w,h_z\rangle_{L^2_{\mathbb{C}}(0,T)}.$$
The inequality in (\ref{normafinita}), the dominated convergence theorem and  Proposition \ref{convergiendo} allow us to compute
\begin{eqnarray}\label{cuentita}
\mathcal{K}^*(z,w;T)&=&\lim_{N,M\rightarrow\infty}\sum_{|n|\leq N}\sum_{ |m|\leq M}\frac{(z+2n)(\overline{w}+2m)}{16\pi}\int_0^T\frac{1}{(T-t)^3} e^{-\frac{(z+2n)^2}{4(T-t)}-\frac{(\overline{w}+2m)^2}{4(T-t)} }dt \notag \\
&=& \lim_{N,M\rightarrow\infty}\sum_{|n|\leq N}\sum_{ |m|\leq M}\mathcal{K}_0(z+2n,w+2m;T)=\mathcal{K}_\ell(z,w;T).
\end{eqnarray}
We also have
\begin{equation}\label{norma1}
\|w(\cdot,T)\|_{\mathcal{H}^\ell_T(D)}=\inf\left\{\|u\|_{L^2_{\mathbb{C}}(0,T)}:w(\cdot,T)=-2\mathcal{L}_T^\ell u, u\in L^2_{\mathbb{C}}(0,T)\right\}.
\end{equation}
\end{proof}

\begin{proof}[Proof of Theorem \ref{derecho}]
We only give a sketch. Let
$\widetilde{h}:-1+D \rightarrow L^2_{\mathbb{C}}(0,T)$ given by
$$\widetilde{h}_z(t)=\overline{(\partial_x\theta)(z+1,T-t)}, \quad t \in (0,T),$$
and the linear operator $\mathcal{L}_T^{r}:L^2_{\mathbb{C}}(0,T)\rightarrow hol(-1+D)$ given by
$$(\mathcal{L}_T^r u)(z)=\langle u, \widetilde{h}_z \rangle_{L^2_{\mathbb{C}}(0,T)}, \quad z\in -1+D.$$
By (\ref{opintegral}) and Remark \ref{perimpar} we have
$$w(z,T)=2(\mathcal{L}_T^r u_r)(z), \quad z\in -1+D,\, u_r\in L^2_{\mathbb{C}}(0,T).$$
Theorem \ref{saitochi} implies that $\mathcal{H}^r_T(-1+D):=\mathcal{R}(\mathcal{L}_T^r)$ is a RKHS on $-1+D$ with reproducing kernel
$$\langle \widetilde{h}_w,\widetilde{h}_z\rangle_{L^2_{\mathbb{C}}(0,T)}=\mathcal{K}_r(z,w;T).$$
We also have
\begin{equation}\label{norma1.5}
\|w(\cdot,T)\|_{\mathcal{H}^r_T(-1+D)}=\inf\left\{\|u\|_{L^2_{\mathbb{C}}(0,T)}:w(\cdot,T)=2\mathcal{L}_T^r u, u\in L^2_{\mathbb{C}}(0,T)\right\}.
\end{equation}
\end{proof}

\begin{remark}
1) Since $\mathcal{K}_\ell(\cdot,w;T)\in \mathcal{H}^\ell_T(D)$ for all $w\in D$  (see \cite[Proposition 2.1, page 71]{saitohnuevo}), we get that the function $\mathcal{K}_\ell(\cdot,y;1) :(0,2)\rightarrow \mathbb{R}$ is in $\mathcal{R}^\ell$ for all $y\in (0,2)$.\\
2) Since $\mathcal{K}_r(\cdot,w;T)\in \mathcal{H}^r_T(D)$ for all $w\in D$, we get that the function $\mathcal{K}_r(\cdot,y;1) :(0,2)\rightarrow \mathbb{R}$ is in $\mathcal{R}^r$ for all $y\in (0,2)$.
\end{remark}

\begin{proof}[Proof of Theorem \ref{laterales}]
\textit{(1)} We set $u_r=-u_\ell$ in (\ref{opintegral}) to get
\begin{eqnarray*}
w(x,T)&=&-2\int_0^T [(\partial_x \theta)(x, T-\tau)+(\partial_x \theta)(x-1, T-\tau)]u_\ell (\tau)d\tau \\
&=& -2\int_0^T (\partial_x \widetilde{\theta})(x, T-\tau)u_\ell (\tau)d\tau
\end{eqnarray*}
where 
$$\widetilde{\theta}(x,t)=\sum_{n\in \mathbb{Z}} K(x+n,t), \quad (x,t)\in \mathbb{R}^2_+.$$
For $t>0$ fixed, clearly the function $(\partial_x\widetilde{\theta})(z,t)$ has similar properties to the analytic theta function $(\partial_x \theta)(z,t)$ in Lemma \ref{estheta}, and also satisfies the following functional equations,
$$f(z+1)=f(z)=-f(-z), \quad z\in Q_\infty.$$
Therefore, $w(\cdot,T)\in hol(Q_\infty)$ and fulfills the last functional equations.\\

Now we proceed as in the proof of Theorem \ref{main}: consider the function $h^+:Q \rightarrow L^2_{\mathbb{C}}(0,T)$ given by

\begin{eqnarray*}
h_z^+(t)&=&\overline{(\partial_x\tilde{\theta})(z,T-t)}\\
&=&\overline{(\partial_x \theta)(z, T-\tau)}+\overline{(\partial_x \theta)(z+1, T-\tau)},\quad t \in (0,T),
\end{eqnarray*}
and the linear operator $\mathcal{L}_T^+:L^2_{\mathbb{C}}(0,T)\rightarrow hol(Q)$ given by
$$(\mathcal{L}_T^+ u)(z)=\langle u, h_z^+ \rangle_{L^2_{\mathbb{C}}(0,T)}, \quad z\in Q.$$
By (\ref{opintegral}) we have
$$w(z,T)=-2(\mathcal{L}_T^+ u_\ell)(z), \quad z\in Q, \,u_\ell\in L^2_{\mathbb{C}}(0,T).$$

Theorem \ref{saitochi} implies that $\mathcal{H}^+_T(D):=\mathcal{R}(\mathcal{L}_T^+)$ is a RKHS on $Q$ with reproducing kernel (the computation is similar to (\ref{cuentita})) 
$$\langle h_w^+,h_z^+\rangle_{L^2_{\mathbb{C}}(0,T)}=\mathcal{K}^+(z,w;T).$$
By the other hand
\begin{eqnarray*}
\langle h_w^+,h_z^+\rangle_{L^2_{\mathbb{C}}(0,T)}&=& \langle h_w,h_z\rangle_{L^2_{\mathbb{C}}(0,T)}+\langle h_{w+1},h_{z+1}\rangle_{L^2_{\mathbb{C}}(0,T)}
+\langle h_{w+1},h_z\rangle_{L^2_{\mathbb{C}}(0,T)}+\langle h_{w},h_{z+1}\rangle_{L^2_{\mathbb{C}}(0,T)}\\
&=& \mathcal{K}_\ell(z,w;T)+\mathcal{K}_\ell(z+1,w+1;T)+\mathcal{K}_\ell(z,w+1;T)+\mathcal{K}_\ell(z+1,w;T), 
\end{eqnarray*}
for $z,w \in Q$, where $h$ is the function in (\ref{achecita}).\\

We also have
\begin{equation}\label{norma2}
\|w(\cdot,T)\|_{\mathcal{H}^+_T(D)}=\inf\left\{\|u\|_{L^2_{\mathbb{C}}(0,T)}:w(\cdot,T)=-2\mathcal{L}_T^+ u, u\in L^2_{\mathbb{C}}(0,T)\right\}.
\end{equation}

\textit{(2)} We set $u_r=u_\ell$ in (\ref{opintegral}) to get
\begin{eqnarray*}
w(x,T)&=& -2\int_0^T [(\partial_x \theta)(x, T-\tau)-(\partial_x \theta)(x-1, T-\tau)]u_\ell (\tau)d\tau\\
&=& -2\int_0^T [(\partial_x \theta)(x, T-\tau)-(\partial_x \theta)(x+1, T-\tau)]u_\ell (\tau)d\tau
\end{eqnarray*}

By Lemma \ref{estheta} and Remark \ref{perimpar} we have
$$(\partial_x \theta)(\cdot, t)-(\partial_x \theta)(\cdot+1, t)\in hol(Q_\infty), \quad \text{for all }t>0,$$
and satisfies the functional equations
$$-f(z+1)=f(z)=-f(-z), \quad z\in Q_\infty.$$
Therefore, $w(\cdot,T)\in hol(Q_\infty)$ and fulfills the last functional equations.\\

Consider the function $h^-:Q \rightarrow L^2_{\mathbb{C}}(0,T)$ given by
$$h_z^-(t)=\overline{(\partial_x \theta)(z, T-t)}-\overline{(\partial_x \theta)(z+1, T-t)}, \quad t \in (0,T),$$
and the linear operator $\mathcal{L}_T^-:L^2_{\mathbb{C}}(0,T)\rightarrow H(Q)$ given by
$$(\mathcal{L}_T^- u)(z)=\langle u, h_z^- \rangle_{L^2_{\mathbb{C}}(0,T)}, \quad z\in Q.$$
Theorem \ref{saitochi} implies that $\mathcal{H}^-_T(D):=\mathcal{R}(\mathcal{L}_T^-)$ is a RKHS on $Q$ with reproducing kernel  
\begin{eqnarray*}
\langle h_w^-,h_z^-\rangle_{L^2_{\mathbb{C}}(0,T)}&=& \langle h_w,h_z\rangle_{L^2_{\mathbb{C}}(0,T)}+\langle h_{w+1},h_{z+1}\rangle_{L^2_{\mathbb{C}}(0,T)}
-\langle h_{w+1},h_z\rangle_{L^2_{\mathbb{C}}(0,T)}-\langle h_{w},h_{z+1}\rangle_{L^2_{\mathbb{C}}(0,T)}\\
&=& \mathcal{K}_\ell(z,w;T)+\mathcal{K}_\ell(z+1,w+1;T)-\mathcal{K}_\ell(z,w+1;T)-\mathcal{K}_\ell(z+1,w;T), 
\end{eqnarray*}
for $z,w \in Q$, where $h$ is the function in (\ref{achecita}).\\

We also have
\begin{equation}\label{norma3}
\|w(\cdot,T)\|_{\mathcal{H}^-_T(D)}=\inf\left\{\|u\|_{L^2_{\mathbb{C}}(0,T)}:w(\cdot,T)=-2\mathcal{L}_T^- u, u\in L^2_{\mathbb{C}}(0,T)\right\}.
\end{equation}
\end{proof}

\begin{remark}
1) Since $\mathcal{K}_+(\cdot,w;T)\in \mathcal{H}^+_T(Q)$ for all $w\in Q$, we get that the function $\mathcal{K}_+(\cdot,y;1) :(0,1)\rightarrow \mathbb{R}$ is in $\mathcal{R}^+$ for all $y\in (0,1)$.\\
2) Since $\mathcal{K}_-(\cdot,w;T)\in \mathcal{H}^-_T(D)$ for all $w\in Q$, we get that the function $\mathcal{K}_-(\cdot,y;1) :(0,1)\rightarrow \mathbb{R}$ is in $\mathcal{R}^-$ for all $y\in (0,1)$.
\end{remark}

\begin{proof}[Proof of Theorem \ref{masmenos}] Let $u_\ell,u_r\in L^2_{\mathbb{C}}(0,T)$. By (\ref{opintegral}) we have
$$-w(z,T)=\mathcal{L}_T^+[u_\ell-u_r](z)+\mathcal{L}_T^-[u_\ell+u_r](z), \quad z\in Q.$$
Therefore $w(\cdot,T)\in \mathcal{H}^+_T(D)+\mathcal{H}^-_T(D)$. Since $ \mathcal{H}^+_T(D)$ and $\mathcal{H}^-_T(D)$ are RKHS on $Q$ with reproducing kernels $\mathcal{K}^+_T(z,w;T)$ and $ \mathcal{K}^-_T(z,w;T)$ respectively, it follows that $\mathcal{H}^+_T(D)+\mathcal{H}^-_T(D)$ is a RKHS with reproducing kernel 
$$\mathcal{K}^+_T(z,w;T)+ \mathcal{K}^-_T(z,w;T)=2\mathcal{K}_\ell(z,w;T)+2\mathcal{K}_\ell(z+1,w+1;T),$$ 
$z,w\in Q$, and is equipped with the norm (see \cite[page 93]{saitohnuevo}) 
\begin{equation}\label{sumaesp}
\|f\|^2=\min\{\|f_1\|^2_{\mathcal{H}^+_T(D)}+\|f_2\|^2_{\mathcal{H}^-_T(D)}:f=f_1+f_2, f_1\in \mathcal{H}^+_T(D),f_2\in \mathcal{H}^-_T(D)\}.
\end{equation}
\end{proof}

\begin{proof}[Proof of Corollary \ref{neumann}]
If $v\in C([0,T];L^2_{\mathbb{C}}(0,1)$ is solution of system (\ref{flujo}), with $u_\ell,u_r\in L^2_{\mathbb{C}}(0,T)$, then $(\partial_x v)(x,t)$ is solution of system (\ref{principal}), therefore $(\partial_x v)(z,t)=\frac{d}{dz}(v(z,t)) \in \mathcal{R}$, $0<t<T$.
\end{proof}

\begin{proof}[Proof of Theorem \ref{iniciando}]
When $u\in C_{\mathbb{C}}((0,T])$ the solution of system (\ref{basico}) is 
\begin{equation}\label{abcd}
v(x,T)=-2\int_0^T  (\partial_x K)(x,T-t)u(t)dt.
\end{equation}

As in the proof of (\ref{fundamental}, case $n=0$), we have that for any compact set $\mathcal{F}\subset \Delta$ and $T>0$, there exist a constant $C_{\mathcal{F},T}>0$
such that 
$$\int_0^T |(\partial_x K)(z,T-t)|^2dt\leq C_{\mathcal{F},T},\quad z\in \mathcal{F}.$$
Hence the integral operator in (\ref{abcd}) is well defined for $u\in L^2_{\mathbb{C}}(0,T)$. Fubini and Moreras's theorems imply that the continuous function
$$z\mapsto \int_0^T  (\partial_x K)(z,T-t)u(t)dt$$
is analytic on $\Delta$.

Consider the function $h^q:\Delta \rightarrow L^2_{\mathbb{C}}(0,T)$ given by
$$h^q_z(t)=\overline{(\partial_x K)(z,T-t)}, \quad t \in (0,T),$$
and the linear operator $\mathcal{L}_T^q:L^2_{\mathbb{C}}(0,T)\rightarrow hol(\Delta)$ given by
$$(\mathcal{L}_T^q u)(z)=\langle u, h_z^q \rangle_{L^2_{\mathbb{C}}(0,T)}, \quad z\in \Delta.$$
Theorem \ref{saitochi} implies that $\mathcal{R}(\mathcal{L}_T^q)$ is a RKHS on $\Delta$ with reproducing kernel 
\begin{eqnarray*}
\mathcal{K}^q(z,w;T)&=&\langle h_w^q,h_z^q\rangle_{L^2_{\mathbb{C}}(0,T)}\\
&=&\frac{z\overline{w}}{16\pi}  \int_0^T \frac{e^{-\frac{z^2}{4(T-t)}}}{(T-t)^{3/2}} \frac{e^{-\frac{\overline{w}^2}{4(T-t)}}}{(T-t)^{3/2}}dt\\
&=&\frac{1}{4\pi} e^{-\frac{z^2}{4T}}e^{-\frac{\overline{w}^2}{4T}}\left(\frac{4z\overline{w}}{ (z^2+\overline{w}^2)^2} + \frac{z\overline{w}}{  T(z^2+\overline{w}^2)}\right).
\end{eqnarray*}

By Remark \ref{pullback} and Corollary 2.5 in \cite[page 107]{saitohnuevo} we have that $e^{-z^2/4}A^2(\Delta)$ is a RKHS with reproducing kernel
$$e^{-\frac{z^2}{4}}e^{-\frac{\overline{w}^2}{4}}\frac{4z\overline{w}}{ (z^2+\overline{w}^2)^2}, \quad \text{and} $$
$$\langle e^{-\frac{z^2}{4}}f,e^{-\frac{z^2}{4}}g\rangle_{e^{-z^2/4}A^2(\Delta)} =\langle f,g\rangle_{A^2(\Delta)}.$$
By Lemma \ref{hardpull} and (\ref{normpull}) we have that $ze^{-z^2/4}\mathcal{H}_{\varphi}(\Delta)$
is a RKHS with reproducing kernel
$$e^{-\frac{z^2}{4}}e^{-\frac{\overline{w}^2}{4}}\frac{z\overline{w}}{ z^2+\overline{w}^2}, \quad \text{and} $$
for all $f,g \in H^2(\mathbb{C}^+)$ we have
$$\langle ze^{-\frac{z^2}{4}}f\circ \varphi,ze^{-\frac{z^2}{4}}g\circ \varphi\rangle_{ze^{-z^2/4}\mathcal{H}_{\varphi}(\Delta)} =\langle f\circ \varphi,g\circ \varphi\rangle_{\mathcal{H}_{\varphi}(\Delta)}=\langle f,g\rangle_{H^2(\mathbb{C}^+)}.$$
Therefore $e^{-z^2/4}A^2(\Delta)+ze^{-z^2/4}\mathcal{H}_{\varphi}(\Delta)$ is the RHKS with reproducing kernel $4\pi\mathcal{K}^q(z,w;1)$, and with norm
\begin{equation}\label{sumanorm}
\|f\|^2_*=\min\{\|e^{\frac{z^2}{4}}f_1\|^2_{A^2(\Delta)}+\|z^{-1} e^{\frac{z^2}{4}}f_2\|^2_{\mathcal{H}_{\varphi}(\Delta)}:f=f_1+f_2, f_1\in e^{-\frac{z^2}{4}}A^2(\Delta),f_2\in ze^{-\frac{z^2}{4}}\mathcal{H}_{\varphi}(\Delta)\}.
\end{equation}
Finally, consider the biholomorphism $\psi :\Delta \rightarrow \Delta $ given by $\psi(z)=T^{-1/2}z$. Hence $\mathcal{K}^q(z,w;T)=T\mathcal{K}^q(\psi(z),\psi(w);1)$ is the reproducing kernel of the space (see Theorem 2.9 in \cite[page 81]{saitohnuevo})
$$\{ f\circ \psi  : f\in  e^{-z^2/4}A^2(\Delta)+ze^{-z^2/4}\mathcal{H}_{\varphi}(\Delta) \}=\mathcal{R}(\mathcal{L}_T^q),$$
with norm 
$$\| f\circ \psi \|_{**}=\| f \|_{*}.$$
Thus $ (\mathcal{R}(\mathcal{L}_T^q),\| \cdot \|_{**})$ is isometrically isomorphic to $(e^{-\frac{z^2}{4}}A^2(\Delta)+ze^{-\frac{z^2}{4}}\mathcal{H}_{\varphi}(\Delta), \|\cdot \|_{*})$ for all $T>0$. As a vector space $\mathcal{R}(\mathcal{L}_T^q)=e^{-\frac{z^2}{4}}A^2(\Delta)+ze^{-\frac{z^2}{4}}\mathcal{H}_{\varphi}(\Delta)$ for all $T>0$.
\end{proof}

\begin{remark}
The integral operator in (\ref{abcd}) is the operator $\widetilde{\Phi}_T$ in \cite{hart,orsini}.  
\end{remark}

\section*{Conclusion}
It is well known that any RKHS is determined by its reproducing kernel, unfortunaly in our cases we cannot write the reproducing kernels in terms of  ``elementary functions". Here we have a fundamental problem: is it posible to found a more manageable description for the RKHSs $(\mathcal{H}^\ell_T(D), \|\cdot \|_{\mathcal{H}_T^\ell(D)}), (\mathcal{H}_T^+(D), \|\cdot \|_{\mathcal{H}_T^+(D)})$ and $(\mathcal{H}_T^-(D), \|\cdot \|_{\mathcal{H}_T^-(D)})$? The proof of Moore's theorem (see \cite[pages 68-71]{saitohnuevo}) gives an equivalent description of  the latter spaces, but it is hard to handle too. We would like to have a description as in Theorem \ref{iniciando}. \\

Our work shows that there is no unique description for the reachable space $\mathcal{R}$ of the heat equation for a finite rod with two Dirichlet boundary controls. It would be interesting to establish a connection between our characterization and the ones obtained in \cite{kellay,orsini}.


\begin{thebibliography}{99}
\bibitem{cannon} Cannon, J. R.; The one-dimensional heat equation.  Encyclopedia of Mathematics and its Applications, 23. Addison-Wesley Publishing Company, Advanced Book Program, Reading, MA, 1984. xxv+483 pp.

\bibitem{darde}  Dard\'e, J\'er\'emi; Ervedoza, Sylvain. On the reachable set for the one-dimensional heat equation. SIAM J. Control Optim. 56 (2018), no. 3, 1692--1715.

\bibitem{duren}  Duren, Peter; Schuster, Alexander Bergman spaces. Mathematical Surveys and Monographs, 100. American Mathematical Society, Providence, RI, 2004. x+318 pp.

\bibitem{fatto} Fattorini, H. O.; Russell, D. L. Exact controllability theorems for linear parabolic equations in one space dimension. Arch. Rational Mech. Anal. 43 (1971), 272–292.

\bibitem{hart} Andreas Hartmann, Karim Kellay, Marius Tucsnak. From the reachable space of the heat equation to Hilbert spaces of holomorphic functions. 2017. ⟨hal-01569695⟩

\bibitem{kellay} Kellay, Karim and Normand, Thomas and Tucsnak, Marius;Sharp reachability results for the heat equation in one space dimension, hal-02302165.

\bibitem{martin} Martin, Philippe; Rosier, Lionel; Rouchon, Pierre On the reachable states for the boundary control of the heat equation. Appl. Math. Res. Express. AMRX 2016, no. 2, 181--216. 

\bibitem{orsini}  Orsini, Marcu-Antone. Reachable states and holomorphic function spaces for the 1-D heat equation. arXiv:1909.01644

\bibitem{saitohnuevo} Saitoh, Saburou; Sawano, Yoshihiro Theory of reproducing kernels and applications. Developments in Mathematics, 44. Springer, Singapore, 2016. xviii+452 pp.



\end{thebibliography}

\end{document}